\documentclass[review,12pt]{article}
\usepackage[margin=2cm]{geometry}
\usepackage{amssymb}
\usepackage{amssymb,amsmath,amsthm,epsfig,multicol}
\usepackage{amsmath}
\usepackage{epsfig}
\usepackage{subfig}
\usepackage{listings}
\usepackage{float}
\usepackage{epstopdf}
\usepackage{xcolor}
\usepackage{color}
\usepackage{mathtools,leftidx}
\definecolor{mygreen}{RGB}{28,172,0}
\usepackage{xcolor}
\usepackage{lineno}
\newtheorem{theorem}{Theorem}[section]
\newtheorem{lemma}[theorem]{Lemma}
\newtheorem{corollary}[theorem]{Corollary}

\theoremstyle{definition}
\newtheorem{definition}[theorem]{Definition}

\newtheorem{example}[theorem]{Example}
\newtheorem{remark}[theorem]{Remark}
\usepackage[noadjust]{cite}

\usepackage{authblk}
\begin{document}
\title{On Laplace transforms with respect to functions and their applications to fractional differential equations}

\author[$\star$,$\dagger$]{Hafiz Muhammad Fahad}
\author[$\star$]{Mujeeb ur Rehman}
\author[$\dagger$]{Arran Fernandez}

\affil[$\star$]{Department of Mathematics, School of Natural Sciences, National University of Sciences and Technology, Islamabad, Pakistan}
\affil[$\dagger$]{Department of Mathematics, Faculty of Arts and Sciences, Eastern Mediterranean University, Famagusta, Northern Cyprus, via Mersin 10, Turkey}

\maketitle

\abstract{An important class of fractional differential and integral operators is given by the theory of fractional calculus with respect to functions, sometimes called $\Psi$-fractional calculus. The operational calculus approach has proved useful for understanding and extending this topic of study. Motivated by fractional differential equations, we present an operational calculus approach for Laplace transforms with respect to functions and their relationship with fractional operators with respect to functions. This approach makes the generalised Laplace transforms much easier to analyse and to apply in practice. We prove several important properties of these generalised Laplace transforms, including an inversion formula, and apply it to solve some fractional differential equations, using the operational calculus approach for efficient solving.}
\section{Introduction}
The birth of fractional calculus finds its roots in the last years of the seventeenth century, when Newton's work along with Leibniz's served as a basis for the inception of classical calculus. Leibniz devised the notation $\frac{\mathrm{d}^n}{\mathrm{d}x^n}f(x)$ to denote the $n$th-order derivative of the function $ f $. When he communicated this to de l'H\^opital, the latter asked about the meaning of the said notation if $ n=\frac{1}{2} $. This communication is unanimously considered as the foundation of fractional calculus. At present, this field has become a matter of deep interest for many researchers.

The most classical operators of fractional calculus are those now called the Riemann--Liouville fractional integral and derivative \cite{Samko,Kilbas}. Caputo made a significant contribution by introducing a different definition of fractional derivatives which is more suitable for certain physical conditions \cite{Caputo,Diethelm}. Moreover, several other families of fractional operators have been introduced and studied until now, out of which Erd\'elyi--Kober, Hadamard, Gr\"unwald--Letnikov, Hilfer, Marchaud, and Prabhakar are just a few to mention \cite{Samko,Kilbas,Hilfer,Prabhakar}.

Due to the large number of definitions of fractional operators, it is important to establish some generalised fractional operators of which the classical ones are particular cases. This enables mathematicians to work on a general level and prove results which can then be used by applied scientists in particular cases \cite{baleanu-fernandez}. Some general classes of fractional operators were proposed in for example \cite{Arran,Zhao,Teodoro}.

One of these classes of generalised fractional operators, which we shall focus on in the current work, is given by applying fractional operators of a function with respect to another function. This concept was introduced in a special case by Erd\'elyi \cite{erdelyi} and in general by Osler \cite{Osler}, and the theory is discussed in some of the standard textbooks such as \cite{Kilbas,Samko}. After the introduction of Riemann--Liouville fractional integrals and derivatives with respect to functions, it is natural to extend the concept to Caputo fractional derivatives with respect to functions \cite{Almeida} and Hilfer fractional derivatives with respect to functions \cite{Sousa}. In the wake of these recent papers, the term ``$\Psi$-fractional calculus'' has gained popularity as a terminology for fractional calculus with respect to a function $\Psi(x)$. Special cases of fractional calculus with respect to functions include Hadamard fractional calculus \cite{hadamard} and the so-called Katugampola fractional calculus \cite{Katugampola}, which in reality is the same as the operator proposed by Erd\`elyi fifty years earlier.

The motivation for writing this paper is largely due to the extensive use of fractional differential equations (FDEs) in physics, economics, engineering and other branches of sciences \cite{Diethelm,Kilbas,Herzallah,Hilfer,Podlubny,Tarasov}. Many different methods exist in the literature for solving FDEs analytically or numerically, and developing some suitable methods to find analytic solutions for some classes of FDEs is one of the most challenging tasks in the field of fractional calculus. In the past few years, some researchers have been interested in introducing fractional extensions of the classical integral transforms, such as Laplace and Fourier transforms, among others \cite{FJAJ,fernandez-baleanu-fokas,Kerr,Zayed,Ozaktas,Namias}. In particular, the Laplace transform and its analogues and extensions \cite{FJAJ,Jarad,Silva} has been considered an effective tool for obtaining analytic solutions to some classes of FDEs. We shall focus specifically on the generalised Laplace transform due to \cite{FJAJ}, which can be called a $\Psi$-Laplace transform or Laplace transform with respect to functions. This operator combines neatly with the fractional integrals and derivatives with respect to functions, and we shall apply it to solve differential equations in this setting.

An important methodology for simplifying fractional calculus with respect to functions, which is mentioned in classical textbooks such as \cite{Samko} but is rarely applied in practice, is the operational calculus approach. This is a way of writing fractional operators with respect to functions as a simple conjugation of the original Riemann--Liouville fractional operators, which is extremely useful in extending knowledge from the Riemann--Liouville setting to the more general setting. Here, we shall apply this methodology to the study of generalised Laplace transforms, which makes many facts about them much easier to prove. It will also be helpful in solving various $\Psi$-fractional differential equations.

The article is organised as follows. Section 2 contains preliminary definitions from classical and fractional calculus. In Section 3, we consider the generalised Laplace transform with respect to functions, including from the viewpoint of operational calculus, and prove several important properties including an inversion formula and generalised Laplace transforms of several fractional operators with respect to functions. Sections 4 and 5 are devoted to fractional differential equations, firstly a regularity result to show applicability of the generalised Laplace transform, and then explicitly solving some Cauchy initial value problems. Finally Section 6 makes concluding statements about this manuscript and future directions of research.
\section{Preliminaries}

Prior to giving the generalised Laplace transform, we first recall some definitions from the classical and fractional calculus.

\allowdisplaybreaks
\subsection{Fractional operators with respect to functions}

\allowdisplaybreaks
\begin{definition}[\cite{Kilbas,Samko,Osler}] \label{defgeneralizedRL}
	Let $ \mu $ be a real number such that $ \mu>0$, $ -\infty \leq a <b \leq \infty  $, $ m= \lfloor \mu \rfloor +1 $, $ f $ be an integrable function defined on $ [a,b] $ and $ \Psi \in C^{1} ([a,b]) $ be an increasing function such that $ \Psi^{\prime}(t) \not = 0 $ for all $ t \in [a,b] $. Then, the $\Psi$-RL fractional integral of $f$ (or Riemann--Liouville fractional integral of $f$ with respect to $\Psi$) of order $\mu$ is defined as
	\begin{equation} \label{psiintegraldef}
	\prescript{}{a}{\mathcal{I}}^{\mu}_{\Psi(t)}f(t) = \frac{1}{\Gamma(\mu)} \int_{a}^{t} \Big( \Psi(t)-\Psi(s) \Big)^{\mu-1} \Psi^\prime (s) f(s)\,\mathrm{d}s,
	\end{equation}
	and the $\Psi$-RL fractional derivative of $f$ (or Riemann--Liouville fractional derivative of $f$ with respect to $\Psi$) of order $\mu$ is defined as
	\begin{equation} \label{psidervativeded}
	\prescript{R}{a}{\mathcal{D}}^{\mu}_{\Psi(t)}f(t) = \left( \frac{1}{\Psi^\prime (t)}\cdot\frac{\mathrm{d}}{\mathrm{d}t} \right)^{m} 	\prescript{}{a}{\mathcal{I}}^{m-\mu}_{\Psi(t)}f(t).
	\end{equation}
\end{definition}

It is to be noted that the case $\Psi(t)=t$ gives $ \prescript{}{a}{\mathcal{I}}^{\mu}_{\Psi(t)}f(t)=\prescript{}{a}{\mathcal{I}}^{\mu}_{t}f(t)$ which is the standard Riemann--Liouville integral. Moreover, for $ \Psi(t)=\log(t) $ the operators defined in \eqref{psiintegraldef}--\eqref{psidervativeded} become the Hadamard fractional integral and derivative respectively.

\begin{definition}[\cite{Almeida}] \label{defgeneralizedC}
	Let $ \mu $ be a real number such that $ \mu>0$, $ -\infty \leq a <b \leq \infty  $, $ m= \lfloor \mu \rfloor +1 $, and $f,\Psi \in C^{m} ([a,b]) $ be functions such that $ \Psi $ is increasing and $ \Psi^{\prime}(t) \not = 0 $ for all $ t \in [a,b] $. Then, the $\Psi$-C fractional derivative of $f$ (or Caputo fractional derivative of $f$ with respect to $\Psi$) of order $\mu$ is defined as
	\begin{equation} \label{psicaputodef}
	\prescript{C}{a}{\mathcal{D}}^{\mu}_{\Psi(t)}f(t) = \prescript{}{a}{\mathcal{I}}^{m-\mu}_{\Psi(t)}\left( \frac{1}{\Psi^\prime (t)} \cdot\frac{\mathrm{d}}{\mathrm{d}t} \right)^{m}f(t).
	\end{equation}
\end{definition}

Again, taking $ \Psi(t)=\log(t) $ or $ \Psi(t)=t $, we get the Caputo-type Hadamard fractional derivative \cite{RAbdeljawadc} and Caputo fractional derivative \cite{Samko} respectively.

The Caputo-type definition \eqref{psicaputodef} is natural to introduce after the Riemann--Liouville-type definitions \eqref{psiintegraldef}--\eqref{psidervativeded} are established. However, it took until Almeida's 2017 paper \cite{Almeida} for this definition to be formalised and analysed. It was that same paper, as far as we can determine, which started the trend of using ``$\Psi$-fractional calculus'' to mean fractional calculus with respect to functions. The next year in 2018, Sousa and Oliveira \cite{Sousa} extended the concept further by combining it with the Hilfer fractional derivatives \cite{Hilfer} to produce the so-called $\Psi$-Hilfer fractional derivative, defined as follows.

\begin{definition} \cite{Sousa}
	Let $ \mu $ be a real number such that $ \mu>0$, $ -\infty \leq a <b \leq \infty  $, $0\leq\nu\leq1$, $ m= \lfloor \mu \rfloor +1 $, and $f,\Psi\in C^{m} ([a,b],\mathbb{R})$ be functions such that $ \Psi $ is increasing and $ \Psi^{\prime}(t) \not = 0 $ for all $ t \in [a,b] $. Then the $\Psi$-Hilfer fractional derivative of $f$ of order $\mu$ and type $\nu$ is given by
	\begin{equation}
	\prescript{}{a}{\mathcal{D}}^{\mu,\nu}_{\Psi(t)}f(t) = \prescript{}{a}{\mathcal{I}}^{\nu(m-\mu)}_{\Psi(t)}\left( \frac{1}{\Psi^\prime (t)}\cdot\frac{\mathrm{d}}{\mathrm{d}t} \right)^{m}  \prescript{}{a}{\mathcal{I}}^{(1-\nu)(m-\mu)}_{\Psi(t)}f(t).
	\end{equation}
\end{definition}

Other fractional operators with respect to functions have also been defined in \cite{Arran,Fahad2,Fahad,Oumarou}, but here we shall focus on the three fundamental ones defined above.

These generalised fractional operators can be written as the conjugation of the standard fractional operators with the operation of composition with  $\Psi$ or $\Psi^{-1}$:
\begin{align}
\prescript{}{a}{\mathcal{I}}^{\mu}_{\Psi(t)}=\mathcal{Q}_\Psi\circ\prescript{}{\Psi(a)}{\mathcal{I}}^{\mu}_{t}\circ \mathcal{Q}_\Psi^{-1},\quad \prescript{R}{a}{\mathcal{D}}^{\mu}_{\Psi(t)}&=\mathcal{Q}_\Psi\circ\prescript{R}{\Psi(a)}{\mathcal{D}}^{\mu}_{t}\circ \mathcal{Q}_\Psi^{-1},\quad\prescript{C}{a}{\mathcal{D}}^{\mu}_{\Psi(t)}=\mathcal{Q}_\Psi\circ   \prescript{C}{\Psi(a)}{\mathcal{D}}^{\mu}_{t}\circ \mathcal{Q}_\Psi^{-1},\label{FwrtF:conjug}\\ \prescript{}{a}{\mathcal{D}}^{\mu,\nu}_{\Psi(t)}&=\mathcal{Q}_\Psi\circ\prescript{}{\Psi(a)}{\mathcal{D}}^{\mu,\nu}_{t}\circ \mathcal{Q}_\Psi^{-1}, \label{FwrtF:Hilfer}
\end{align}
where the functional operator $\mathcal{Q}_\Psi$ is defined by
\begin{equation}
\label{Qdef}
(\mathcal{Q}_\Psi f)(x)=f(\Psi(x)).
\end{equation}
These conjugation expressions all interconnect with each other since fractional derivatives are generally defined using composition. They are mentioned in \cite{Samko} and are very useful in proving various properties of fractional calculus with respect to functions, since many results in this generalised setting can now be proved directly from the corresponding properties in the classical Riemann--Liouville or Caputo setting.

\subsection{Laplace transforms and some special functions}

\begin{definition} 
	Assume that the function $ f $ is defined for $ t \geq 0 $. Then the Laplace transform of $f$, denoted by $\mathcal{L}_{} \left\{ f \right\}$, is defined by the improper integral
	\begin{equation} \label{classicallaplacedef}
	\mathcal{L}_{} \left\{ f(t) \right\}=F(s)=\int_{0}^{\infty} e^{-s t} f(t)\,\mathrm{d}t
	\end{equation}
	provided that the integral in \eqref{classicallaplacedef} exists, i.e., that the integral is convergent. 
	
	The inverse Laplace transform is defined by
	\begin{equation} \label{inverselaplace}
	\mathcal{L}^{-1} \left\{F(s) \right\} =\frac{1}{{2 \pi i}} \int_{c-i\infty}^{c+i\infty} e^{st }F(s)\,\mathrm{d}s.
	\end{equation}
\end{definition}

\begin{lemma}[\cite{Podlubny,Z. Tomovski}] We recall the Laplace transforms of the basic fractional operators.
	\begin{itemize}
		\item[(a)]  The Laplace transform of the Riemann–Liouville fractional integral of order $ \mu $ is given by: 
		\begin{equation} \label{LTRLI}
		\mathcal{L}_{} \left\{ (\prescript{}{0}{\mathcal{I}}^{\mu}_{t}f)(t) \right\}= s^{-\mu}	\mathcal{L}_{} \left\{ f(t) \right\}.
		\end{equation}
		\item[(b)] The Laplace transform of the Riemann–Liouville fractional derivative of order $ \mu $ is given by: 
		\begin{equation}\label{LTRLD}
		\mathcal{L}_{} \left\{  \prescript{R}{0}{\mathcal{D}}^{\mu}_{t}f(t) \right\} =s^{\mu} \mathcal{L}_{} \left\{ f(t) \right\} - \sum_{i=0}^{m-1} s^{m-i-1}(\prescript{}{0}{\mathcal{I}}^{m-i-\mu}_{t}f)(0).
		\end{equation}
		\item[(c)] The Laplace transform of the Caputo fractional derivative of order $ \mu $ is given by:  
		\begin{equation} \label{LTCD}
		\mathcal{L}_{} \left\{  \prescript{C}{0}{\mathcal{D}}^{\mu}_{t}f(t) \right\} =s^{\mu} \mathcal{L}_{} \left\{ f(t) \right\} - \sum_{i=0}^{m-1} s^{\mu-i-1}(\prescript{}{0}{\mathcal{D}}^{i}_{t}f)(0).
		\end{equation}
		\item[(d)] The Laplace transform of the Hilfer fractional derivative of order $ \mu $ and type $ \nu $ is given by:  
		\begin{equation} \label{LTHD} 
		\mathcal{L}_{} \left\{  \prescript{}{0}{\mathcal{D}}^{\mu,\nu}_{t}f(t) \right\} =s^{\mu} \mathcal{L}_{} \left\{ f(t) \right\} - \sum_{i=0}^{m-1} s^{m(1-\nu)+\mu \nu-i-1}(\prescript{}{0}{\mathcal{I}}^{(1-\nu)(m-\mu)-i}_{t}f)(0).
		\end{equation}
	\end{itemize}
\end{lemma}

%

There are several special functions which are considered to be helpful for finding the solutions of FDEs. In the following Definitions, we present a few of them.
\begin{definition}
	The Wright function is defined as
	\begin{equation}
	\label{Wrightdefn}
	W(z, \mu, \nu) : = \sum_{j=0}^{\infty} \frac{z^{j}}{j! \Gamma(\mu j+ \nu)}, \qquad \mu,\nu,z\in\mathbb{C}, \mathrm{Re}(\mu) > -1,
	\end{equation}
	which is an entire function of $z$ on the whole complex plane. It appeared for the first time in \cite{Wright1,Wright2} in connection with E. M. Wright's investigations in the asymptotic theory of partitions.
\end{definition}

\begin{definition}
	In \cite{Mittag}, G\"osta Mittag-Leffler introduced the well-known Mittag-Leffler function	$E_{\mu} (z)$, given by 
	\begin{equation} \label{mlfun}
	E_{\mu} (z)= \sum_{j=0}^{\infty} \frac{z^{j}}{\Gamma(\mu j+ 1)}, \ \ \mu \in \mathbb{C}, \ \  \text{Re}(\mu)  > 0.
	\end{equation}
	Later on, a natural generalisation of $E_{\mu} (z)$ was discussed by Wiman in \cite{Wiman}. He introduced the function $E_{\mu,\nu} (z)$ as
	\begin{equation} \label{wiman}
	E_{\mu,\nu} (z)= \sum_{j=0}^{\infty} \frac{z^{j}}{\Gamma(\mu j+ \nu)}, \ \ \mu, \nu \in \mathbb{C}, \ \  \text{Re}(\mu) > 0.
	\end{equation}
	If we consider $ \nu=1 $ in \eqref{wiman}, we obtain the Mittag-Leffler function \eqref{mlfun}. In \cite{Prabhakar}, Prabhakar presented a more generalised version of \eqref{mlfun}-\eqref{wiman}, called the Prabhakar function or three-parameter Mittag-Leffler function, which is defined by the series representation
	\begin{equation} \label{Prabhakar}
	E_{\mu,\nu}^{\gamma} (z) : = \frac{1}{\Gamma(\gamma)}\sum_{j=0}^{\infty} \frac{\Gamma(\gamma+j)z^{j}}{j! \Gamma(\mu j+ \nu)}, \ \ \mu, \nu, \gamma \in \mathbb{C}, \ \  \text{Re}(\mu) > 0.
	\end{equation}
	This is an entire function of $z$ of order $1/\text{Re}(\mu)$. This function plays a necessary role in the explanation of the anomalous dielectric properties in heterogeneous systems, and some important properties of this function can be seen in \cite{Kilbas1,Nigmatullin,Saxena,Garra}.
\end{definition}

The Mittag-Leffler and related functions defined above have some interesting connections with the Laplace transform, as summarised in the following Lemma which will be used later.

\begin{lemma}[\cite{gorenflo-kilbas-mainardi-rogosin}] \label{Lem:LapML}
	The standard Mittag-Leffler function has a Laplace transform given by
	\[
	\mathcal{L}\left\{E_{\mu}(\lambda t^\mu)\right\}=\frac{s^{\mu-1}}{s^{\mu}-\lambda},
	\]
	while the three-parameter version has a Laplace transform given by
	\[
	\mathcal{L}\left\{t^{\nu-1}E_{\mu,\nu}^{\gamma}(\lambda t^\mu)\right\}=\frac{s^{\mu\gamma-\nu}}{\left(s^{\mu}-\lambda\right)^\gamma}.
	\]
\end{lemma}


\section{The generalised Laplace transform}

In this section, we discuss a generalised integral transform introduced by Jarad and Abdeljawad \cite{FJAJ} which can be used to solve linear FDEs involving $\Psi$-RL, $\Psi$-C and $\Psi$-Hilfer fractional derivatives. This new integral transform is the natural generalisation of the classical Laplace transform for the setting of fractional operators with respect to functions. Here we develop the operational calculus approach which makes the generalised Laplace transform, and its relationship with generalised fractional operators, much easier to understand. Using this approach together with classical results on the original Laplace transform, we prove several important properties of the generalised Laplace transform. The derivation of an inversion formula also forms part of this section.

\begin{definition}[\cite{FJAJ}] \label{psilaplacedef}
	Let $ f:[0,\infty) \to \mathbb{R} $ be a real-valued function and $ \Psi $ be a non-negative increasing function such that $\Psi(0)=0$. Then the Laplace transform of $f$ with respect to $\Psi$ is defined by
	\begin{equation} \label{def}
	\mathcal{L}_{\Psi} \left\{ f(t) \right\}=F(s)=\int_{0}^{\infty} e^{-s \Psi(t)} \Psi^\prime (t) f(t) \,\mathrm{d}t
	\end{equation}
	for all $s\in\mathbb{C}$ such that this integral converges. Here $\mathcal{L}_{\Psi}$ denotes the Laplace transform with respect to $\Psi$, which we call a generalised Laplace transform.
\end{definition}

\begin{definition}[\cite{FJAJ}] \label{Def:psiexp}
	An $ n $-dimensional function $\boldsymbol{f}:[0,\infty)\to\mathbb{R}^n$ is said to be of $ \Psi $-exponential order $c>0$ if there exist positive constants $ M $ and $T$ such that for all $ t>T $,
	\[
	\left\|\boldsymbol{f}\right\|_{\infty} {=} \max_{1 \leq i \leq n} \left\|f_i \right\|_{\infty}\leq M e^{c \Psi(t)},
	\]
	i.e. if
	\[
	\boldsymbol{f}(t)= \mathcal{O}( e^{c \Psi(t)})  \text{ \ \ \ \ as \ \   }  t \to \infty.
	\]		
\end{definition}

Now we state some sufficient conditions for the existence of the generalised Laplace transform of a function. This motivates the inclusion of Definition \ref{Def:psiexp} here.

\begin{theorem} \label{existence}
	If $ f:[0,\infty) \to \mathbb{R} $ is a piecewise continuous function and is of $ \Psi $-exponential order $c>0$, where $\Psi$ is a non-negative increasing function with $\Psi(0)=0$, then the generalised Laplace transform of $f$ exists for $ s > c $.
\end{theorem}

\begin{proof}
	For the $\Psi$-exponentially bounded function $f$, we have
	\begin{align*}
	\Big |\mathcal{L}_{\Psi} \left\{ f(t) \right\}\Big | &= \left|\int_{0}^{\infty} e^{-s \Psi(t)} \Psi'(t) f(t)\,\mathrm{d}t	\right| \\
	&\leq \int_{0}^{\infty} e^{-s \Psi(t)} \Psi'(t) \Big |f(t)\Big | \,\mathrm{d}t \\
	&\leq M \int_{0}^{\infty} e^{-s \Psi(t)} \Psi'(t) e^{c \Psi(t)}  \,\mathrm{d}t \\
	&=M\left[\frac{e^{-(s-c)\Psi(t)}}{-(s-c)}\right]_0^{\infty} \\
	&= \frac{M}{s-c}, 
	\end{align*}
	where in the last line we used the assumption that $s>c$. Therefore, the integral is convergent and the proof of the Theorem is complete.
\end{proof}

\begin{remark}
	From the above proof, it follows that $ \lim_{s \to \infty} 	\mathcal{L}_{\Psi} \left\{ f(t) \right\} = 0 $ for any $\Psi$-exponentially bounded function $f$. This property can be called the limiting property of the generalised Laplace transform.
\end{remark}

In the following theorem, we prove a new relationship between the classical Laplace transform and generalised Laplace transform. This is motivated by the operational calculus approach to fractional calculus with respect to functions.

\begin{theorem}	\label{GLT:theorem}
	The generalised Laplace transform may be written as a combination of the classical Laplace transform with the operation of composition with $\Psi$ or $\Psi^{-1}$, as follows:
	\begin{equation} \label{GLT:comp}
	\mathcal{L}_\Psi=\mathcal{L}\circ \mathcal{Q}_\Psi^{-1},
	\end{equation}
	where the functional operator $\mathcal{Q}_\Psi$ is defined by \eqref{Qdef} as before.
\end{theorem}

\begin{proof}
	We here verify that these two operators are the same, by starting with a function $f$ and examining the action of both operators:
	\begin{align*}
	f:t&\mapsto f(t); \\
	\mathcal{Q}_\Psi^{-1}(f):t&\mapsto f(\Psi^{-1}(t)); \\
	\mathcal{L}\circ \mathcal{Q}_\Psi^{-1}(f):t&\mapsto \int_{0}^{\infty} e^{-s t} f(\Psi^{-1}(t))\,\mathrm{d}t=\int_{0}^{\infty} e^{-s \Psi(u)} \Psi'(u) f(u)\,\mathrm{d}u,
	\end{align*}
	whereas
	\begin{align*}
	\mathcal{L}_\Psi \{f\}:t&\mapsto \int_{0}^{\infty} e^{-s \Psi(t)} \Psi'(t) f(t)\,\mathrm{d}t.
	\end{align*}
	Thus the result is clearly seen.
\end{proof}

The result of Theorem \ref{GLT:theorem} will be very useful in understanding the generalised Laplace transform, because many standard results on the classical Laplace transform now extend directly to the corresponding results on the generalised Laplace transform, simply by composition. Here we are following the principle that the easiest way to prove properties of a new mathematical object may be to write it in terms of an old one. We shall see that all the results of \cite{FJAJ} are now very easy to prove by using Theorem \ref{GLT:theorem}, but first we start with two direct natural corollaries of Theorem \ref{GLT:theorem}.

\begin{corollary}\label{IGLT:cor}
	The inverse generalised Laplace transform may be written as a combination of the inverse classical Laplace transform with the operation of composition with $\Psi$ or $\Psi^{-1}$, as follows:
	\begin{equation}
	\label{IGLT:comp}
	\mathcal{L}_\Psi^{-1}=\mathcal{Q}_\Psi \circ \mathcal{L}^{-1},
	\end{equation}
	or in other words
	\begin{equation} \label{inverseGL}
	\mathcal{L}^{-1}_{\Psi} \left\{ F(s)  \right\} =\frac{1}{{2 \pi i}} \int_{c-i\infty}^{c+i\infty} e^{s\Psi(t) }  F(s) \,\mathrm{d}s.
	\end{equation}
\end{corollary}

\begin{corollary}\label{Coroll:genfn}
	If $f(t)$ is a function whose classical Laplace transform is $F(s)$, then the generalised Laplace transform of the function $f\circ\Psi=f(\Psi(t))$ is also $F(s)$:
	\[
	\mathcal{L}\{f(t)\}=F(s)\quad\Rightarrow\quad\mathcal{L}_{\Psi}\{f(\Psi(t))\}=F(s).
	\]
\end{corollary}

\begin{example}[\cite{FJAJ}]
	\begin{itemize}
		\item[(a)]  $ \mathcal{L}_{\Psi} \left\{ (\Psi(t))^{\mu} \right\}=\frac{\Gamma(\mu+1)}{s^{\mu+1}}, \text{\ \ \ \ for \ \ } s>0. $
		\item[(b)]  $ \mathcal{L}_{\Psi} \left\{ e^{a \Psi(t)} \right\} =\frac{1}{s-a}, \text{\ \ \ \ for \ \ } s>a. $
		\item[(c)]  $ \mathcal{L}_{\Psi} \left\{ E_{\mu}\Big(\lambda ({\Psi(t)})^{\mu}\Big)   \right\} =\frac{s^{\mu-1}}{s^{\mu}-\lambda}, \text{\ \ \ \ for \ \ }  \text{Re}(\mu)>0 \text{\ \ \ \ and \ \ }\Big|\frac{\lambda}{s^{\mu}} \Big|<1. $
		\item[(d)]  $ \mathcal{L}_{\Psi} \left\{ ({\Psi(t)})^{\mu-1} E_{\mu,\mu}\Big(\lambda ({\Psi(t)})^{\mu}\Big) \right\} =\frac{1}{s^{\mu}-\lambda}, \text{\ \ \ \ for \ \ }  \text{Re}(\mu)>0 \text{\ \ \ \ and \ \ }\Big|\frac{\lambda}{s^{\mu}} \Big|<1. $
	\end{itemize}
\end{example}

\allowdisplaybreaks
\begin{example}
	Assume that $ \text{Re}(\mu)>0$ and $ \Big|\frac{\lambda}{s^{\mu}} \Big|<1$. If $ E_{\mu,\nu}^{\gamma} $ denotes the Prabhakar function \eqref{Prabhakar}, then we have
	\begin{align*}
	\mathcal{L}_{\Psi} \left\{ ({\Psi(t)})^{\nu-1} E_{\mu,\nu}^{\gamma}\Big(\lambda ({\Psi(t)})^{\mu}\Big) \right\} &=\mathcal{L}_{} \left\{ t^{\nu-1} E_{\mu,\nu}^{\gamma}\left(\lambda t^{\mu}\right) \right\}=  \frac{s^{\mu\gamma-\nu}}{(s^{\mu}-\lambda)^{\gamma}}.
	\end{align*}	
\end{example}

Both of the above Examples can be proved easily using Theorem \ref{GLT:theorem} together with results on standard Laplace transforms of functions such as Lemma \ref{Lem:LapML}.


In the following theorem, we give a direct proof for the uniqueness of the generalised Laplace transform and an inversion formula for it.

\begin{theorem}
	Assume that $f$ and $g$ are piecewise continuous functions on $[0,\infty)$ and of $ \Psi $-exponential order $c>0$. If $\mathcal{L}_{\Psi}\{f\}(s)=\mathcal{L}_{\Psi}\{g\}(s)$ for $ s>c $, then $ f(t)=g(t) $ for all $t\geq0$.
\end{theorem}

\begin{proof}
	For the uniqueness, since the generalised Laplace transform is linear, it will be sufficient to prove that if $F(s)= \mathcal{L}_{\Psi}\{f(t)\}(s)=0 $ for all $ s>c $ then $ f(t)=0 $ for all $ t \geq 0 $.
	
	Fixing $ s_{0} > c $ and making the substitution $ u= e^{- \Psi(t)} $ in the integral definition \eqref{psilaplacedef}, for $ s = s_{0} + n +1 $ we get
	\begin{equation} \label{prooffirstequation}
	\begin{aligned}
	0={F}(s)&=\int_{0}^{\infty} e^{-s_{0} \Psi(t)} e^{-n \Psi(t)}e^{- \Psi(t)} \Psi^\prime (t) f(t) dt= \int_{0}^{1} u^{n} \left\{ u^{s_{0}} f\left( \Psi^{-1} \left( -\log u \right) \right) \right\} \,\mathrm{d}u
	\end{aligned}
	\end{equation}
	where $ n=0,1,2,\dots $. Let us define $ r(u)=  u^{s_{0}} f\left( \Psi^{-1} \left( -\log u \right) \right) $, which is a piecewise continuous function on $ (0,1] $ satisfying
	\[
	\lim_{u \to 0}r(u)=\lim_{t \to \infty} e^{-s_{0} \Psi(t)} f(t)=0.
	\]
	If we let $ r(0)=0 $, then $ r $ is a piecewise continuous function satisfying  
	\begin{equation}\label{proof2ndequation}
	\int_{0}^{1} p(u) r(u) \,\mathrm{d}u =0 
	\end{equation}
	where $p$ is any polynomial. Thus, if any function $ \hat{r} $ has a power series expansion which converges uniformly on $ [0,1] $, then Eq. \eqref{proof2ndequation} can be rewritten as 
	\begin{equation}\label{proof3rdequation}
	\int_{0}^{1} \hat{r}(u) r(u) \,\mathrm{d}u =0.
	\end{equation}
	
	Suppose that $r$ is not the zero function; then we can find a point $ u_{0} \in (0,1) $, an interval $ I=[u_{0}-c_{0} ,u_{0}+c_{0} ]  \subset [0,1]$, and a constant $ a $ such that $ r(u) \geq a >0 $ for all $ u \in I $. Assume that $ \hat{r}(u)= \frac{1}{b} e^{-\left(\frac{u-u_{0}}{b} \right)^{2}} $, where $ b>0 $; then $ \hat{r} $ has a power series expansion which converges uniformly on $ [0,1] $, so that Eq. \eqref{proof3rdequation} holds. Thus for $ x=\frac{u-u_{0}}{b}$, we have:
	\begin{align*}
	J_1 &= \int_{u_{0}-c_{0}}^{u_{0}+c_{0}} \hat{r}(u) du = \int_{-c_{0}/b}^{c_{0}/b} e^{-x^{2}} \,\mathrm{d}x, \\
	J_2 &= \int_{u_{0}+c_{0}}^{1} \hat{r}(u) du = \int_{c_{0}/b}^{\left(1-u_{0}\right)/b} e^{-x^{2}} \,\mathrm{d}x, \\
	J_3 &= \int_{0}^{u_{0}-c_{0}} \hat{r}(u) du = \int_{-u_{0}/b}^{-c_{0}/b} e^{-x^{2}} \,\mathrm{d}x.
	\end{align*}
	Since $\int_{-\infty}^{\infty} e^{-x^{2}} \,\mathrm{d}x =\sqrt{\pi}$, then for any sufficiently small $ \epsilon >0 $, there is a $ \delta > 0 $ such that if $ 0 < b \leq \delta $ we have
	\[
	J_1 \geq \frac{\sqrt{\pi}}{2}, \qquad 0 \leq J_2 \leq \epsilon, \qquad 0 \leq J_3 \leq \epsilon.
	\]
	Since $ r(u) \geq a >0 $ for all $ u \in I $ and $\left| r \right| < M $ for some $ M < \infty $, we have
	\[
	\int_I \hat{r}(u) r(u) \,\mathrm{d}u \geq \frac{\sqrt{\pi}a}{2} >0,\ \ \ \ \  \left|  \int_{[0,1]\backslash I} \hat{r}(u) r(u) \,\mathrm{d}u \right|\leq 2 M \epsilon, 
	\]
	and hence 
	\[
	\int_{0}^{1}\hat{r}(u) r(u) \,\mathrm{d}u \geq \frac{\sqrt{\pi}a}{2} - 2 M \epsilon >0,
	\]
	provided $ \epsilon < \frac{\sqrt{\pi}a}{4 M} $, contradicting Eq. \eqref{proof3rdequation}. Thus, $ r $ is the zero function which implies that $ f$ is the zero function. This completes the proof.
\end{proof}

\subsection{Combination with $\Psi$-fractional operators}

In the following theorems, we find the generalised Laplace transforms of the Riemann--Liouville, Caputo, and Hilfer fractional operators with respect to functions (so-called $\Psi$-RL, $ \Psi$-C and $\Psi$-Hilfer operators). Some of the results were already shown in \cite{FJAJ}, but the proofs here are more straightforward, using the operational calculus approach introduced in this paper along with the known Laplace transforms of standard fractional operators.

\begin{theorem}[\cite{FJAJ}] \label{inttrans}
	Let $ \mu >0 $ and let $f$ be a function of $\Psi$-exponential order, piecewise continuous over each finite interval $ [0,T] $. Then
	\begin{equation} \label{integraltrans}
	\mathcal{L}_{\Psi} \left\{ \left(\prescript{}{0}{\mathcal{I}}^{\mu}_{\Psi(t)}f\right)(t) \right\}= s^{-\mu}\mathcal{L}_{\Psi} \left\{ f(t) \right\}.
	\end{equation}
\end{theorem}

\begin{proof} We prove the result by combining Theorem \ref{GLT:theorem} with the known relations \eqref{LTRLI} and \eqref{FwrtF:conjug}:
	\[
	\mathcal{L}_{\Psi}\circ\prescript{}{0}{\mathcal{I}}^{\mu}_{\Psi(t)} = \left(\mathcal{L}\circ \mathcal{Q}_\Psi^{-1}\right)\circ\left(\mathcal{Q}_\Psi\circ\prescript{}{0}{\mathcal{I}}^{\mu}_{t}\circ \mathcal{Q}_\Psi^{-1}\right)=\left(\mathcal{L}\circ\prescript{}{0}{\mathcal{I}}^{\mu}_{t}\right)\circ\mathcal{Q}_\Psi^{-1},
	\]
	therefore, writing $g=\mathcal{Q}_\Psi^{-1}f$ so that $f(t)=g(\Psi(t))$, we have:
	\begin{align*}
	f:t&\mapsto g(\Psi(t)); \\
	\mathcal{Q}_\Psi^{-1}f:t&\mapsto g(t); \\
	\mathcal{L}_{\Psi}\circ\prescript{}{0}{\mathcal{I}}^{\mu}_{\Psi(t)}f:t&\mapsto \left(\mathcal{L}\circ\prescript{}{0}{\mathcal{I}}^{\mu}_{t}\right)g(t)=s^{-\mu}\mathcal{L}\{g(t)\},
	\end{align*}
	where $\mathcal{L}\{g(t)\}=\mathcal{L}_{\Psi} \left\{ f(t) \right\}$ by Corollary \ref{IGLT:cor}.
\end{proof}

\begin{theorem}[\cite{FJAJ}] \label{theoremderivative}
	Assume that $ \mu>0 $, $ m= \lfloor\mu\rfloor +1 $, and $f$ is a function such that $ f(t) $, $ \prescript{}{0}{\mathcal{I}}^{m-\mu}_{\Psi(t)}f(t) $, $ \prescript{}{0}{\mathcal{D}}^{1}_{\Psi(t)} \prescript{}{0}{\mathcal{I}}^{m-\mu}_{\Psi(t)}f(t)$, $\ldots$, $\prescript{}{0}{\mathcal{D}}^{m}_{\Psi(t)} \prescript{}{0}{\mathcal{I}}^{m-\mu}_{\Psi(t)}f(t) $ are continuous on $(0,\infty)$ and of $\Psi$-exponential order, while $ \prescript{R}{0}{\mathcal{D}}^{\mu}_{\Psi(t)}f(t) $ is piecewise continuous on $[0,\infty)$. Then
	\begin{equation}
	\mathcal{L}_{\Psi}\left\{\left(\prescript{R}{0}{\mathcal{D}}^{\mu}_{\Psi(t)}f\right)(t) \right\} = s^{\mu}\mathcal{L}_{\Psi}\{f(t)\} - \sum_{i=0}^{m-1} s^{m-i-1}\left(\prescript{}{0}{\mathcal{I}}^{m-i-\mu}_{\Psi(t)}f\right)(0).
	\end{equation}
\end{theorem}

\begin{proof}
	Just like in the previous Theorem, we can show directly that
	\[
	\mathcal{L}_{\Psi}\circ\prescript{R}{0}{\mathcal{D}}^{\mu}_{\Psi(t)}=\left(\mathcal{L}\circ\prescript{R}{0}{\mathcal{D}}^{\mu}_{t}\right)\circ\mathcal{Q}_\Psi^{-1}.
	\]
	Then, again writing $g=\mathcal{Q}_\Psi^{-1}f$ so that $f(t)=g(\Psi(t))$, and using the relation \eqref{LTRLD}, we have:
	\begin{align*}
	f:t&\mapsto g(\Psi(t)); \\
	\mathcal{Q}_\Psi^{-1}f:t&\mapsto g(t); \\
	\mathcal{L}_{\Psi}\circ\prescript{R}{0}{\mathcal{D}}^{\mu}_{\Psi(t)}f:t&\mapsto \left(\mathcal{L}\circ\prescript{R}{0}{\mathcal{D}}^{\mu}_{t}\right)g(t)= s^{\mu} \mathcal{L}\{g(t)\} - \sum_{i=0}^{m-1} s^{m-i-1}\left(\prescript{}{0}{\mathcal{I}}^{m-i-\mu}_{t}g\right)(0).
	\end{align*}
	We know $\mathcal{L}\{g(t)\}=\mathcal{L}_{\Psi} \left\{ f(t) \right\}$ by Corollary \ref{IGLT:cor}. Also $\prescript{}{0}{\mathcal{I}}^{m-i-\mu}_{t}g=\left(\prescript{}{0}{\mathcal{I}}^{m-i-\mu}_{t}\circ\mathcal{Q}_\Psi^{-1}\right)f$, and this, together with \eqref{FwrtF:conjug} and the assumption that $\Psi(0)=0$, gives the result in the desired form.
\end{proof}

\begin{theorem}[\cite{FJAJ}]
	Assume that $ \mu>0 $, $ m= \lfloor\mu\rfloor+1 $, and $f$ is a function such that $ f(t) $, $\prescript{}{0}{\mathcal{D}}^{1}_{\Psi(t)} f(t) $, $ \prescript{}{0}{\mathcal{D}}^{2}_{\Psi(t)} f(t)$, $\ldots$, $\prescript{}{0}{\mathcal{D}}^{m-1}_{\Psi(t)}f(t) $ are continuous on $[0,\infty)$ and of $\Psi$-exponential order, while $   \prescript{C}{0}{\mathcal{D}}^{\mu}_{\Psi(t)}f(t) $ is piecewise continuous on $[0,\infty)$. Then
	\begin{equation}
	\mathcal{L}_{\Psi} \left\{ \left(\prescript{C}{0}{\mathcal{D}}^{\mu}_{\Psi(t)}f\right)(t) \right\} =s^{\mu} \mathcal{L}_{\Psi}\{f(t)\} - \sum_{i=0}^{m-1} s^{\mu-i-1}\left(\prescript{}{0}{\mathcal{D}}^{i}_{\Psi(t)}f\right)(0).
	\end{equation}
\end{theorem}

\begin{proof}
	Similarly to the previous Theorem, using the identity \eqref{FwrtF:conjug} together with the result of Theorem \ref{GLT:theorem} and the relation \eqref{LTCD}.
\end{proof}

\begin{theorem}
	Assume $ \mu>0 $, $ m= \lfloor\mu\rfloor+1 $, $ 0 \leq \nu \leq 1 $, and $f$ is a function such that $ f(t) $, $ \prescript{}{0}{\mathcal{D}}^{j}_{\Psi(t)} \prescript{}{0}{\mathcal{I}}^{(1-\nu)(m-\mu)}_{\Psi(t)}f(t) \in  C[0,\infty)$ are of $\Psi$-exponential order for $j=0,1,2,\cdots, m-1$, while $\prescript{}{0}{\mathcal{D}}^{\mu,\nu}_{\Psi(t)}f(t) $ is piecewise continuous on $ [0,\infty) $. Then
	\begin{equation}
	\mathcal{L}_{\Psi} \left\{  \prescript{}{0}{\mathcal{D}}^{\mu,\nu}_{\Psi(t)}f(t) \right\} =s^{\mu} \mathcal{L}_{\Psi}\{f(t)\} - \sum_{i=0}^{m-1} s^{m(1-\nu)+\mu \nu-i-1}\left(\prescript{}{0}{\mathcal{I}}^{(1-\nu)(m-\mu)-i}_{\Psi(t)}f\right)(0).
	\end{equation}	
\end{theorem}

\begin{proof}
	Combining Theorem \ref{GLT:theorem} with the relation \eqref{FwrtF:Hilfer}, we find as before
	\[
	\mathcal{L}_{\Psi}\circ\prescript{}{0}{\mathcal{D}}^{\mu,\nu}_{\Psi(t)} = \left(\mathcal{L}\circ \mathcal{Q}_\Psi^{-1}\right)\circ\left(\mathcal{Q}_\Psi\circ\prescript{}{0}{\mathcal{D}}^{\mu,\nu}_{t}\circ \mathcal{Q}_\Psi^{-1}\right)=\left(\mathcal{L}\circ\prescript{}{0}{\mathcal{D}}^{\mu,\nu}_{t}\right)\circ\mathcal{Q}_\Psi^{-1},
	\]
	therefore, again writing $g=\mathcal{Q}_\Psi^{-1}f$ so that $f(t)=g(\Psi(t))$, and using \eqref{LTHD}, we have:
	\begin{align*}
	f:t&\mapsto g(\Psi(t)); \\
	\prescript{}{0}{\mathcal{D}}^{\mu,\nu}_{t}\circ\mathcal{Q}_\Psi^{-1}f:t&\mapsto \prescript{}{0}{\mathcal{D}}^{\mu,\nu}_{t}g(t); \\
	\mathcal{L}_\Psi \prescript{}{0}{\mathcal{D}}^{\mu,\nu}_{\Psi(t)}f:t&\mapsto \left(\mathcal{L}\circ\prescript{}{0}{\mathcal{D}}^{\mu,\nu}_{t}\right)g(t)=s^{\mu} \mathcal{L}\{g(t)\} - \sum_{i=0}^{m-1} s^{m(1-\nu)+\mu \nu-i-1}(\prescript{}{0}{\mathcal{I}}^{(1-\nu)(m-\mu)-i}_{t}g)(0). 
	\end{align*}
	We know $\mathcal{L}\{g(t)\}=\mathcal{L}_{\Psi} \left\{ f(t) \right\}$ by Corollary \ref{IGLT:cor}. Finally, rewriting the fractional integral initial values in the same way as in Theorem \ref{theoremderivative}, we get the result in the desired form.
\end{proof}

\subsection{Combination with $\Psi$-convolution}

In this section, we consider a type of generalised convolution operation between two functions, which can be called $\Psi$-convolution, and which ties together naturally with the $\Psi$-Laplace transform. Again, the operational calculus approach enables easy proofs for many properties of the $\Psi$-convolution operation.

\begin{definition} \cite{FJAJ}
	Let $f$ and $g$ be of $ \Psi $-exponential order, piecewise continuous functions over each finite interval $ [0,T] $. Then, the $ \Psi $-convolution of $f$ and $g$ is the function $f\ast_\Psi g$ defined by 
	\begin{equation}\label{convolution}
	\big(f*_{\Psi} g\big)(t)= \int_{0}^{t} f\Big( \Psi^{-1} (\Psi(t)-\Psi(\tau))\Big) g(\tau) \Psi'(\tau) \,\mathrm{d}\tau.
	\end{equation}
\end{definition}

In the following theorem, we find a relationship between classical convolution and generalised convolution, using the operational calculus approach.

\begin{theorem}	\label{GLCandLCtheorem}
	Let $f$ and $g$ be of $ \Psi $-exponential order, piecewise continuous functions over each finite interval $ [0,T] $. Then, the following relation holds:
	\[
	f*_{\Psi} g= \mathcal{Q}_\Psi\left((\mathcal{Q}_\Psi^{-1}f) * (\mathcal{Q}_\Psi^{-1}g) \right),
	\]
	or in other words
	\[
	f*_{\Psi} g= \Big((f\circ\Psi^{-1}) * (g\circ\Psi^{-1}) \Big)\circ\Psi,
	\]
	which can also be rewritten as
	\begin{equation}
	\label{convol:reln}
	(f\circ\Psi)*_\Psi(g\circ\Psi)=(f*g)\circ\Psi.
	\end{equation}
\end{theorem}

\begin{proof}
	By the definition of classical convolution, and substituting $ \tau = \Psi(u)$ inside the integral, we have
	\begin{align*}
	\left(f * g \right)(t)&= \int_{0}^{t} f(t-\tau) g(\tau) \,\mathrm{d}\tau \\
	&= \int_{0}^{\Psi^{-1}(t)} f\left(t-\Psi(u)\right) g\left(\Psi(u)\right) \Psi ^{\prime} (u)\,\mathrm{d}u.
	\end{align*}
	Applying $ \mathcal{Q}_\Psi $, we find
	\begin{align*}
	(f * g)\circ\Psi(t)&=\int_{0}^{t} f\left(\Psi(t)-\Psi(u)\right) g\left(\Psi(u)\right) \Psi ^{\prime} (u)\,\mathrm{d}u \\
	&=\int_{0}^{t} (f\circ\Psi)\Big( \Psi^{-1} (\Psi(t)-\Psi(u))\Big) (g\circ\Psi)(u) \Psi'(u) \,\mathrm{d}u.
	\end{align*}
	This completes the proof.
\end{proof}

If we interpret convolution (both classical and generalised) as a binary operation acting on two functions, and use the alternative notation $*(f,g)$ and $*_\Psi(f,g)$ instead of $f*g$ and $f*_\Psi g$, then the result of Theorem \ref{GLCandLCtheorem} can be written, like Eq. \eqref{FwrtF:conjug} for $\Psi$-fractional derivatives and integrals, as a conjugation of operators:
\[
*_\Psi = \mathcal{Q}_\Psi\circ * \circ\left(\mathcal{Q}_\Psi^{-1},\mathcal{Q}_\Psi^{-1}\right),
\]
namely
\[
*_\Psi(f,g) = \mathcal{Q}_\Psi\Big(*\left(\mathcal{Q}_\Psi^{-1}f,\mathcal{Q}_\Psi^{-1}g\right)\Big).
\]
In the same way as the conjugation results for fractional operators, this operational representation of $\Psi$-convolution enables many properties to be proved easily. In particular, we consider commutativity, associativity, and distributivity, some of which were seen in \cite{FJAJ} but can now be proved much more easily.

\begin{theorem} 
	Let $f$ and $g$ and $h$ be of $ \Psi $-exponential order, piecewise continuous functions over each finite interval $ [0,T] $, and let $a$ and $b$ be constants. Then
	\begin{enumerate}
		\item[(a)] $f *_{\Psi} g = g *_{\Psi} f$. 
		\item[(b)] $ \left( f *_{\Psi} g \right) *_{\Psi} h = f*_{\Psi} \left(g *_{\Psi} h \right)$. 
		\item[(c)] $ f *_{\Psi} \left(a g + b h\right) = a f*_{\Psi} g + b f*_{\Psi} h$.
	\end{enumerate}
\end{theorem}
\allowdisplaybreaks

\begin{proof}
	This theorem can be proved by using the result of Theorem \ref{GLCandLCtheorem} together with the corresponding results on classical convolution.  So we omit the straightforward details.
\end{proof}

\begin{remark}
	The set of all generalised Laplace transformable functions forms a commutative semigroup with respect to the binary operation $ *_{\Psi} $. But it does not form a group, because $ f^{-1} *_{\Psi} g $ is not generalised Laplace transformable in general. 
\end{remark}

One of the most important properties of convolution in connection with integral transforms is that the Laplace transform of the convolution of two functions is the product of their Laplace transforms. In the next result, we prove the corresponding theorem for the generalised Laplace transform of the $ \Psi $-convolution of two functions.

\begin{theorem}
	Assume that $f$ and $g$ are piecewise continuous functions on $ [0,T] $ and of $ \Psi $-exponential order $c>0 $. Then,
	\begin{equation}
	\mathcal{L}_\Psi \left\{ f*_{\Psi} g\right\}=\mathcal{L}_\Psi\{f\}\mathcal{L}_\Psi\{g\}.
	\end{equation}
\end{theorem}

\begin{proof}
	We prove this theorem by using the result of Theorem \ref{GLCandLCtheorem} together with Theorem \ref{GLT:theorem}:
	\[
	\left(\mathcal{L}_\Psi\right)\circ\left(*_\Psi\right)=\left(\mathcal{L}\circ\mathcal{Q}_\Psi^{-1}\right)\circ\left(\mathcal{Q}_\Psi\circ * \circ\left(\mathcal{Q}_\Psi^{-1},\mathcal{Q}_\Psi^{-1}\right)\right)=\left(\mathcal{L}\circ*\right)\circ\left(\mathcal{Q}_\Psi^{-1},\mathcal{Q}_\Psi^{-1}\right),
	\]
	therefore
	\[
	\left(\mathcal{L}_\Psi\right)\circ\left(*_\Psi\right)(f,g)=\left(\mathcal{L}\circ*\right)\left(\mathcal{Q}_\Psi^{-1}f,\mathcal{Q}_\Psi^{-1}g\right)=\mathcal{L}\left(\mathcal{Q}_\Psi^{-1}f\right)\mathcal{L}\left(\mathcal{Q}_\Psi^{-1}g\right)=\mathcal{L}_\Psi\{f\}\mathcal{L}_\Psi\{g\}.
	\]
	Hence, we get our desired result.
\end{proof}

\section{A regularity result for $\Psi$-fractional differential equations}

In this section, we examine the effectiveness of the generalised Laplace transform method for solving fractional-order differential equations of the following type:
\begin{align}
\prescript{C}{0}{\mathcal{D}}^{\mu}_{\Psi(t)}\boldsymbol{y}(t)&=A\boldsymbol{y}(t)+\boldsymbol{g}(t), \qquad 0<\mu<1, \qquad t\geq0, \label{de}
\\ \boldsymbol{y}(0)&= \boldsymbol{\eta},  \label{ic}
\end{align}
where $ \prescript{C}{0}{\mathcal{D}}^{\mu}_{\Psi(t)} $ is the Caputo-type fractional differential operator with respect to $\Psi(t)$, $ A = \left(a_{ij}\right) $ is an $n\times n$ constant matrix, and $ \boldsymbol{g}$ is an $n$-dimensional continuous function.

\begin{theorem} \label{theorem5point1}
	Assume that the system $\eqref{de}-\eqref{ic}$ has a unique and continuous solution $ \boldsymbol{y}(t) $. If the forcing function $ \boldsymbol{g}(t)$ is continuous on $[0,\infty) $ and $ \Psi $-exponentially bounded, then the functions $ \boldsymbol{y}(t) $ and $ \prescript{C}{0}{\mathcal{D}}^{\mu}_{\Psi(t)}\boldsymbol{y}(t) $ are both $ \Psi $-exponentially bounded too.
\end{theorem}

\begin{proof}
	It can be noticed that the initial value problem $\eqref{de}-\eqref{ic}$ is equivalent to the following Volterra equation:
	\begin{equation} \label{volterraeq}
	\boldsymbol{y}(t)=\boldsymbol{\eta} + \frac{1}{\Gamma(\mu)} \int_{0}^{t} \left( \Psi(t)- \Psi(\tau) \right)^{\mu-1} \Psi'(\tau) \left\{ A\boldsymbol{y}(\tau)+\boldsymbol{g}(\tau) \right\} \,\mathrm{d}\tau, \qquad 0\leq t<\infty.
	\end{equation}
	By assumption, $ \boldsymbol{g}(t) $ is $ \Psi $-exponentially bounded, so there exist positive constants $ c $, $ M $ and large enough $ T $ such that $ \|\boldsymbol{g}(t)\|_{\infty}\leq Me^{c\Psi(t)} $ for all $ t \geq T$. Now, for $ t \geq T$, \eqref{volterraeq} can be written as
	\begin{multline} 
	\boldsymbol{y}(t)=\boldsymbol{\eta} + \frac{1}{\Gamma(\mu)} \int_{0}^{T} \left( \Psi(t)- \Psi(\tau) \right)^{\mu-1} \Psi'(\tau) \left\{ A\boldsymbol{y}(\tau)+\boldsymbol{g}(\tau) \right\} \,\mathrm{d}\tau\\+ \frac{1}{\Gamma(\mu)} \int_{T}^{t} \left( \Psi(t)- \Psi(\tau) \right)^{\mu-1} \Psi'(\tau) \left\{ A\boldsymbol{y}(\tau)+\boldsymbol{g}(\tau) \right\}\,\mathrm{d}\tau.
	\end{multline}
	Since $ \boldsymbol{y}(t) $ is the unique and continuous solution of $\eqref{de}-\eqref{ic}$ on $ [0,\infty) $, we know that $ A\boldsymbol{y}(t)+\boldsymbol{g}(t) $ is bounded on $ [0,T] $. That is, there exists a constant $\ell>0$ such that $\|A\boldsymbol{y}(t)+\boldsymbol{g}(t)\|_{\infty}<\ell$. So,
	\begin{align*} 
	\|\boldsymbol{y}(t)\|_{\infty} \leq \|\boldsymbol{\eta}\|_{\infty} &+ \frac{\ell}{\Gamma(\mu)} \int_{0}^{T} \left( \Psi(t)- \Psi(\tau) \right)^{\mu-1} \Psi'(\tau) \,\mathrm{d}\tau\\
	&+ \frac{1}{\Gamma(\mu)} \int_{T}^{t} \left( \Psi(t)- \Psi(\tau) \right)^{\mu-1} \Psi'(\tau) \|A\|_{\infty} \ \|\boldsymbol{y}(\tau)\|_{\infty}\,\mathrm{d}\tau \\
	&+ \frac{1}{\Gamma(\mu)}\int_{T}^{t} \left(\Psi(t)-\Psi(\tau) \right) ^{\mu-1} \Psi'(\tau) \|\boldsymbol{g}(\tau)\|_{\infty} \,\mathrm{d}\tau,
	\end{align*}
	where the $ \infty $-norm of the matrix $ A $ is defined as $\|A\|_{\infty}= \max_{1 \leq i \leq n} \sum_{j=1}^{n} |a_{ij}| $. Using the facts that $ e^{-c\Psi(t)} \leq  e^{-c\Psi(T)}$ and $ e^{-c\Psi(t)} \leq  e^{-c\Psi(\tau)}$ and $\|\boldsymbol{g}(t)\|_{\infty} \leq Me^{c\Psi(t)} $, and multiplying the above inequality by $ e^{-c\Psi(t)} $, we find
	\allowdisplaybreaks
	\begin{align*}
	\|\boldsymbol{y}(t)\|_{\infty}e^{-c\Psi(t)} &\leq \|\boldsymbol{\eta}\|_{\infty}e^{-c\Psi(t)} + \frac{\ell e^{-c\Psi(t)}}{\Gamma(\mu)} \int_{0}^{T} \left( \Psi(t)- \Psi(\tau) \right)^{\mu-1} \Psi'(\tau) \,\mathrm{d}\tau\\
	&\hspace{2cm}+ \frac{e^{-c\Psi(t)}}{\Gamma(\mu)} \int_{T}^{t} \left( \Psi(t)- \Psi(\tau) \right)^{\mu-1} \Psi'(\tau) \|A\|_{\infty} \ \|\boldsymbol{y}(\tau)\|_{\infty}\,\mathrm{d}\tau \\
	&\hspace{2cm}+ \frac{e^{-c\Psi(t)}}{\Gamma(\mu)} \int_{T}^{t} \left( \Psi(t)- \Psi(\tau) \right)^{\mu-1} \Psi'(\tau) \|\boldsymbol{g}(\tau)\|_{\infty} \,\mathrm{d}\tau \\
	&\leq \|\boldsymbol{\eta}\|_{\infty}e^{-c\Psi(T)} + \frac{le^{-c\Psi(T)}}{\mu\Gamma(\mu)}  \Big( (\Psi(t))^{\mu}- (\Psi(t)-\Psi(T))^{\mu} \Big) \\
	&\hspace{2cm}+ \frac{\|A\|_{\infty}}{\Gamma(\mu)} \int_{0}^{t} \left( \Psi(t)- \Psi(\tau) \right)^{\mu-1} \Psi'(\tau)  \ \|\boldsymbol{y}(\tau)\|_{\infty} e^{-c\Psi(\tau)} \,\mathrm{d}\tau \\
	&\hspace{2cm}+ \frac{M}{\Gamma(\mu)} \int_{0}^{t} \left( \Psi(t)- \Psi(\tau) \right)^{\mu-1} \Psi'(\tau) e^{c(\Psi(\tau)-\Psi(t))}  \,\mathrm{d}\tau \\
	&\leq \|\boldsymbol{\eta}\|_{\infty}e^{-c\Psi(T)}+\frac{\ell(\Psi(T))^{\mu} e^{-c\Psi(T)}}{\Gamma(\mu+1)} + \frac{M}{\Gamma(\mu)} \int_{0}^{\infty} e^{-c s} s^{\mu-1} \,\mathrm{d}s \\
	&\hspace{2cm}+ \frac{\|A\|_{\infty}}{\Gamma(\mu)} \int_{0}^{t} \left( \Psi(t)- \Psi(\tau) \right)^{\mu-1} \Psi'(\tau)  \ \|\boldsymbol{y}(\tau)\|_{\infty} e^{-c\Psi(\tau)} \,\mathrm{d}\tau \\
	&\leq \|\boldsymbol{\eta}\|_{\infty}e^{-c\Psi(T)} + \frac{\ell(\Psi(T))^{\mu} e^{-c\Psi(T)}}{\Gamma(\mu+1)}+ \frac{M}{c^{\mu}}  \\
	&\hspace{2cm}+ \frac{\|A\|_{\infty}}{\Gamma(\mu)} \int_{0}^{t} \left( \Psi(t)- \Psi(\tau) \right)^{\mu-1} \Psi'(\tau)  \ \|\boldsymbol{y}(\tau)\|_{\infty} e^{-c\Psi(\tau)} \,\mathrm{d}\tau.
	\end{align*}
	Define positive constants $a$, $b$ by
	\[
	a=\|\boldsymbol{\eta}\|_{\infty}e^{-c\Psi(T)} + \frac{\ell(\Psi(T))^{\mu} e^{-c\Psi(T)}}{\Gamma(\mu+1)}+ \frac{M}{c^{\mu}}, \qquad  b=\frac{\|A\|_{\infty}}{\Gamma(\mu)},\]
	and then we have
	\[
	\|\boldsymbol{y}(t)\|_{\infty}e^{-c\Psi(t)} \leq a+b \int_{0}^{t} \left( \Psi(t)- \Psi(\tau) \right)^{\mu-1} \Psi'(\tau) \ \|\boldsymbol{y}(\tau)\|_{\infty} e^{-c\Psi(\tau)} \,\mathrm{d}\tau. \nonumber
	\]
	Using the Gr\"onwall inequality proved in \cite[Corollary 2]{Vanterler da C. Sousa}, we deduce
	\begin{align} \label{eq10}
	\|\boldsymbol{y}(t)\|_{\infty}e^{-c\Psi(t)} &\leq a E_{\mu} \Big(b\Gamma(\mu)\left( \Psi(t)-\Psi(0) \right)^{\mu} \Big) = a E_{\mu} \Big( \|A\|_{\infty} \left( \Psi(t) \right)^{\mu} \Big). 
	\end{align}
	For $ 0< \mu<2 $, $ u>0 $, $ t \geq 0 $, the following inequality is shown in \cite[Eq. (2.9)]{bazhlekova}:
	\begin{equation} \label{eq11}
	E_{\mu} \Big( u \left( \Psi(t) \right)^{\mu} \Big) \leq C e^{u^{{1}/{\mu}}\Psi(t)},
	\end{equation}
	where $C>0$ is independent of $t$ and $u$. From \eqref{eq10} and \eqref{eq11}, we have
	\[
	\|\boldsymbol{y}(t)\|_{\infty}e^{-c\Psi(t)} \leq a C e^{(\|A\|_{\infty})^{{1}/{\mu}}\Psi(t)},
	\]
	which means
	\[
	\|\boldsymbol{y}(t)\|_{\infty} \leq a C e^  { \left\{(\|A\|_{\infty})^{{1}/{\mu}}+c  \right\} \Psi(t)}.	
	\]
	Thus, $ \boldsymbol{y}(t) $ is $ \Psi $-exponentially bounded. Moreover, from Eq. \eqref{de}, we have
	\begin{align*}
	\|\prescript{C}{0}{\mathcal{D}}^{\mu}_{\Psi(t)}\boldsymbol{y}(t)\|_{\infty} & \leq \|A\|_{\infty} \ \|\boldsymbol{y}(t)\|_{\infty} +\|\boldsymbol{g}(t)\|_{\infty} \\& \leq a \|A\|_{\infty} C e^  { \left\{(\|A\|_{\infty})^{{1}/{\mu}}+c  \right\} \Psi(t)} +M e^{c \Psi(t)} \\& \leq \Big( a \|A\|_{\infty} C+ M \Big) e^{\left(\|A\|_{\infty}^{1/\mu}+c \right) \Psi(t)}.
	\end{align*} 
	Thus, $\prescript{C}{0}{\mathcal{D}}^{\mu}_{\Psi(t)}\boldsymbol{y}(t)$ is also $ \Psi $-exponentially bounded, and this completes the proof.
\end{proof}

Similar results can be proved for fractional-order differential equations in the settings of $\Psi$-RL and $\Psi$-Hilfer fractional derivatives. Moreover, for the case $ \Psi(t)=t $, an analogue of Theorem \ref{theorem5point1} has been seen to hold true in \cite{Kexue}.
\section{Applications to fractional differential equations}

In the present section, by using the operational calculus approach developed in this paper for generalised Laplace transforms, we find solutions of some different classes of linear FDEs with constant coefficients, in the settings of $\Psi$-RL, $\Psi$-C and $\Psi$-Hilfer fractional derivatives.

It is well known that sometimes fractional differential equations are more appropriate to model certain systems than classical ones. The generalisation from integer order to fractional order allows for a wider range of behaviours to be described. The same principle applies to the generalisation from operators with respect to $t$ to operators with respect to $\Psi(t)$. For example, Almeida \cite{Almeida} showed that a population growth model could be reproduced more accurately by considering different possible functions $\Psi$.

\subsection{Solutions of some non-homogeneous linear $ \Psi$-RL and $ \Psi$-C FDEs}

\begin{theorem} \label{app1RLFDE}
	For $0<\mu\leq1$ and $\lambda,c\in\mathbb{R}$, the fractional initial value problem
	\begin{align}
	\prescript{R}{0}{\mathcal{D}}^{\mu}_{\Psi(t)}y(t) - \lambda y(t)&=f(t),  \label{app1de} \\
	\prescript{}{0}{\mathcal{I}}^{(1-\mu)}_{\Psi(t)}y(0)&=c, \label{app1ic}
	\end{align} 
	has the solution
	\begin{equation}\label{GRLFDE}
	y(t)=c({\Psi(t)})^{\mu-1} E_{\mu,\mu}\Big(\lambda ({\Psi(t)})^{\mu}\Big)+({\Psi(t)})^{\mu-1} E_{\mu,\mu}\Big(\lambda ({\Psi(t)})^{\mu}\Big) *_{\Psi}f(t).
	\end{equation}
\end{theorem}

\begin{proof}
	Using the identity \eqref{FwrtF:conjug}, the IVP \eqref{app1de}--\eqref{app1ic} can be transformed into the following Riemann--Liouville IVP:
	\begin{align*}
	\prescript{R}{0}{\mathcal{D}}^{\mu}_{t}z(t) - \lambda z(t)&=g(t), \\
	\prescript{}{0}{\mathcal{I}}^{(1-\mu)}_{t}y(0)&=c,
	\end{align*}
	where $z=\mathcal{Q}_\Psi^{-1}y$ and $g=\mathcal{Q}_\Psi^{-1}f$ i.e. $y=z\circ\Psi$ and $f=g\circ\Psi$. This classical Riemann--Liouville initial value problem has the solution
	\begin{equation}\label{app1Speicalsolution}
	z(t)=c{t}^{\mu-1} E_{\mu,\mu}\big(\lambda t^{\mu}\big)+t^{\mu-1} E_{\mu,\mu}\big(\lambda t^{\mu}\big) * g(t).
	\end{equation}
	Applying $ \mathcal{Q}_\Psi $ to both sides of \eqref{app1Speicalsolution}, we get \eqref{GRLFDE} as required.
\end{proof}

\begin{theorem} \label{app2CFDE}
	For $0<\mu\leq1$ and $\lambda,c\in\mathbb{R}$, the fractional initial value problem
	\begin{align}
	\prescript{C}{0}{\mathcal{D}}^{\mu}_{\Psi(t)}y(t) - \lambda y(t)&=f(t), \label{app2de} \\
	y(0)&=c,  \label{app2ic}
	\end{align} 
	has the solution
	\begin{equation} \label{app2answer}
	y(t)=c E_{\mu}\Big(\lambda ({\Psi(t)})^{\mu}\Big)+({\Psi(t)})^{\mu-1} E_{\mu,\mu}\Big(\lambda ({\Psi(t)})^{\mu}\Big) *_{\Psi}f(t).
	\end{equation}
\end{theorem}

\begin{proof}
	The proof of Theorem \ref{app2CFDE} is similar to that of Theorem \ref{app1RLFDE}. We choose to omit the details involved.
\end{proof}

\begin{corollary}
	For $0<\mu\leq1$, the following special case of the IVP \eqref{app2de}--\eqref{app2ic}:
	\begin{align*}
	\prescript{C}{0}{\mathcal{D}}^{\mu}_{\Psi(t)}y(t) -  y(t)&=1, \\
	y(0)&=1,
	\end{align*}
	has the following solutions for some specific choices of function $\Psi$:
	\begin{enumerate}
	\item[(a)] If $\Psi(t)=\sqrt{t}$, then $ y(t)= E_{\mu}(t^{\frac{\mu}{2}})+t^{\frac{\mu}{2}} E_{\mu,\mu+1}( t^{\frac{\mu}{2}})$.
	\item[(b)] If $\Psi(t)=t$, then $ y(t)= E_{\mu}(t^{\mu})+t^{\mu} E_{\mu,\mu+1}( t^{\mu})$.
	\item[(c)] If $\Psi(t)=t^2$, then $ y(t)= E_{\mu}(t^{2\mu})+t^{2\mu} E_{\mu,\mu+1}( t^{2\mu})$.
\end{enumerate}
\end{corollary}

\begin{proof}
	We consider case (b), since the manipulations for (a) and (c) are very similar.
	
	From \eqref{convolution} and \eqref{app2answer}, we have
	\begin{align*}
	y(t)&= E_{\mu}(t^{\mu})+ \int_{0}^{t} {\tau}^{\mu-1} E_{\mu,\mu}( {\tau}^{\mu}) \,\mathrm{d}\tau = E_{\mu}(t^{\mu})+ \int_{0}^{t} \sum_{k=0}^{\infty}  \frac{{\tau}^{\mu k + \mu-1}}{\Gamma(\mu k + \mu)} \,\mathrm{d}\tau \\&=E_{\mu}(t^{\mu})+ \sum_{k=0}^{\infty}  \frac{{t}^{\mu k + \mu}}{\Gamma(\mu k + \mu+1)}=E_{\mu}(t^{\mu})+t^{\mu} E_{\mu,\mu+1}( t^{\mu}),
	\end{align*}
	which is the desired result.
\end{proof}


\begin{theorem}
	Let $0<\mu\leq1$. The fractional diffusion equation 
	\begin{equation} \label{app3de} 
	\prescript{R}{0}{\mathcal{D}}^{\mu}_{\Psi(t)}u(x,t) = \kappa \frac{\partial^{2}u(x,t)}{\partial  x^{2}},
	\end{equation}
	with initial and boundary conditions
	\begin{align}
	u(x,t) &\to 0 \quad\mathrm{ as }\quad |x| \to \infty,  \label{app3c1} \\
	\prescript{}{0}{\mathcal{I}}^{(1-\mu)}_{\Psi(t)}u(x,t) \Big|_{t=0}&=f(x), \qquad x \in \mathbb{R},  \label{app3c2}
	\end{align} 
	has the solution
	\begin{equation}\label{app3sol}
	u(x,t)= \int_{-\infty}^{\infty} G(x-\eta,t)f(\eta)\,\mathrm{d}\eta,
	\end{equation}
	where $G$ is defined by
	\[
	G(x,t)= \frac{1}{2\sqrt{\kappa}} (\Psi(t))^{\frac{\mu}{2}-1} W \Big( - \frac{|x|}{\sqrt{\kappa}(\Psi(t))^{\frac{\mu}{2}}}, -\frac{\mu}{2}, \frac{\mu}{2} \Big),
	\]
	and $W$ is the Wright function defined in \eqref{Wrightdefn} above.
\end{theorem}

\begin{proof}
	Let us define $v=\mathcal{Q}_\Psi^{-1}u$, where the operator $\mathcal{Q}_\Psi$ is applied with respect to the second variable $t$, so that $u(x,t)=v(x,\Psi(t))$. Then, once again using the identity \eqref{FwrtF:conjug}, the FDE \eqref{app3de} can be transformed into the following classical Riemann--Liouville FDE:
	\begin{equation} \label{app3Sde}
	\prescript{R}{0}{\mathcal{D}}^{\mu}_{t}v(x,t) = \kappa \frac{\partial^{2}v(x,t)}{\partial  x^{2}},
	\end{equation}
	while the initial and boundary conditions \eqref{app3c1}--\eqref{app3c2} can similarly be transformed into the following initial and boundary conditions for \eqref{app3Sde}:
	\begin{align}
	v(x,t) &\to 0 \quad\mathrm{ as }\quad |x| \to \infty, \label{app3Sc1} \\
	\prescript{}{0}{\mathcal{I}}^{(1-\mu)}_{t}v(x,t) \Big|_{t=0}&=f(x), \qquad x \in \mathbb{R}, \label{app3Sc2}
	\end{align}
	where for transforming the initial condition we used the fact that $\Psi(0)=0$.
	
	The FDE \eqref{app3Sde} with initial and boundary conditions \eqref{app3Sc1}--\eqref{app3Sc2} has the solution (see \cite[\S6.7]{debnath-bhatta}):
	\begin{equation}\label{app3Specialsol}
	v(x,t)= \int_{-\infty}^{\infty} \frac{1}{2\sqrt{\kappa}} t^{\frac{\mu}{2}-1} W \left( -\frac{|x-\eta|}{\sqrt{\kappa}t^{\frac{\mu}{2}}}, -\frac{\mu}{2}, \frac{\mu}{2} \right)f(\eta)d\eta,
	\end{equation}
	or in other words
	\[
	\mathcal{Q}_\Psi^{-1}u=\left(\mathcal{Q}_\Psi^{-1}G\right)*f=\mathcal{Q}_\Psi^{-1}\left(G*f\right),
	\]
	where the operators $\mathcal{Q}_\Psi$ are applied with respect to the second variable $t$ but the convolution is with respect to the first variable $x$. Applying $ \mathcal{Q}_\Psi $ to both sides, we get the required solution \eqref{app3sol}.
\end{proof}

It can be noted that for $ \Psi(t)=t $ and $ \mu=1 $, the Cauchy problem \eqref{app3de}-\eqref{app3c2} reduces to the classical diffusion problem and the solution \eqref{app3sol} reduces to the classical fundamental solution.

\subsection{Solutions of some general $ \Psi$-Hilfer FDEs}

In this section, we consider some multi-term fractional ordinary differential equations using $\Psi$-Hilfer derivatives. Equations of a similar type using plain Hilfer derivatives (in other words the special case $\Psi(t)=t$) were previously studied in \cite{Z. Tomovski}. Here we apply the operational calculus approach to efficiently find the general solutions to some initial value problems of $\Psi$-Hilfer type, which have been demonstrated in \cite{HilferExperimental} to have applications e.g. in the study of dielectric relaxation in glasses.

\begin{theorem}
	Assume that $0< \mu_1 \leq \mu_2< 1$ and $0\leq\nu_j\leq1$, $a_j\in \mathbb{R}$ for $j=1,2$. Consider the $\Psi$-Hilfer FDE
	\begin{equation}\label{app4de}
	a_{1}\prescript{}{0}{\mathcal{D}}^{\mu_1,\nu_{1}}_{\Psi(t)}y(t) +a_{2}\prescript{}{0}{\mathcal{D}}^{\mu_2,\nu_{2}}_{\Psi(t)}y(t) + a_{3} y(t)=f(t),
	\end{equation} 
	with initial conditions
	\begin{equation} \label{app4ic}
	\prescript{}{0}{\mathcal{I}}^{(1-\nu_j)(1-\mu_{j})}_{\Psi(t)} y(0)=b_{j}, \qquad j=1,2.
	\end{equation} 
	The initial value problem \eqref{app4de}--\eqref{app4ic} has the solution
	\begin{multline*}
	y(t) = \frac{1}{a_2} \sum_{k=0}^{\infty} \Big( -\frac{a_1}{a_2} \Big)^{k} \left[f(t)*_\Psi\left((\Psi(t))^{(\mu_{2}-\mu_{1})k+\mu_{2}-1} E_{\mu_{2},(\mu_{2}-\mu_{1})k+\mu_{2}}^{k+1} \left( - \frac{a_3}{a_2} (\Psi(t))^{\mu_{2}} \right)\right) \right. \\
	\left. +  a_2 b_2 (\Psi(t))^{(\mu_{2}-\mu_{1})k+\mu_{2}+\nu_2 (1-\mu_{2})-1} E_{\mu_{2},(\mu_{2}-\mu_{1})k+\mu_{2}+\nu_2 (1-\mu_{2})}^{k+1} \Big( - \frac{a_3}{a_2} (\Psi(t))^{\mu_{2}} \Big)   \right. \\
	\left. + a_1 b_1 (\Psi(t))^{(\mu_{2}-\mu_{1})k+\mu_{2}+\nu_1 (1-\mu_{1})-1} E_{\mu_{2},(\mu_{2}-\mu_{1})k+\mu_{2}+\nu_1 (1-\mu_{1})}^{k+1} \Big( - \frac{a_3}{a_2} (\Psi(t))^{\mu_{2}} \Big)  \right].
	\end{multline*}
\end{theorem}
\allowdisplaybreaks

\begin{proof}
	Again write $z=\mathcal{Q}_\Psi^{-1}y$ and $g=\mathcal{Q}_\Psi^{-1}f$ so that $y=z\circ\Psi$ and $f=g\circ\Psi$. Then using the identity \eqref{FwrtF:Hilfer}, the IVP \eqref{app4de}--\eqref{app4ic} can be transformed into the following IVP:
	\begin{align*}
	a_{1}\prescript{}{0}{\mathcal{D}}^{\mu_1,\nu_{1}}_{t}z(t) + a_{2}\prescript{}{0}{\mathcal{D}}^{\mu_2,\nu_{2}}_{t}z(t) + a_{3} z(t)=g(t), \\
	\prescript{}{0}{\mathcal{I}}^{(1-\nu_j)(1-\mu_{j})}_ty(0)=b_{j}, \qquad j=1,2.
	\end{align*}
	which was solved in \cite[Theorem 5]{Z. Tomovski}. The solution to this Hilfer FDE is
	\begin{multline*}
	z(t) = \frac{1}{a_2} \sum_{k=0}^{\infty} \Big( -\frac{a_1}{a_2} \Big)^{k} \left[g(t)*\left(t^{(\mu_{2}-\mu_{1})k+\mu_{2}-1} E_{\mu_{2},(\mu_{2}-\mu_{1})k+\mu_{2}}^{k+1} \Big( - \frac{a_3}{a_2} t^{\mu_{2}} \Big)\right) \right. \\
	\left. +  a_2 b_2 t^{(\mu_{2}-\mu_{1})k+\mu_{2}+\nu_2 (1-\mu_{2})-1} E_{\mu_{2},(\mu_{2}-\mu_{1})k+\mu_{2}+\nu_2 (1-\mu_{2})}^{k+1} \Big( - \frac{a_3}{a_2} t^{\mu_{2}} \Big)   \right. \\
	\left.+a_1 b_1 t^{(\mu_{2}-\mu_{1})k+\mu_{2}+\nu_1 (1-\mu_{1})-1} E_{\mu_{2},(\mu_{2}-\mu_{1})k+\mu_{2}+\nu_1 (1-\mu_{1})}^{k+1} \Big( - \frac{a_3}{a_2} t^{\mu_{2}} \Big)  \right].
	\end{multline*}
	Applying $ \mathcal{Q}_\Psi $ to both sides of the above equation, and using the result of Theorem \ref{GLCandLCtheorem} on $\Psi$-convolution, we get the required result.
\end{proof}

\begin{theorem}
	Assume that $0< \mu_1 \leq \mu_2 \leq \mu_3< 1$ and $0 \leq \nu_j\leq 1$, $a_j\in \mathbb{R}$ for $j=1,2,3$. Then the $\Psi$-Hilfer initial value problem
	\begin{align*}
	a_{1}\prescript{}{0}{\mathcal{D}}^{\mu_1,\nu_{1}}_{\Psi(t)}y(t) + a_{2}\prescript{}{0}{\mathcal{D}}^{\mu_2,\nu_{2}}_{\Psi(t)}y(t) &+a_{3}\prescript{}{0}{\mathcal{D}}^{\mu_3,\nu_{3}}_{\Psi(t)}y(t) + a_{4} y(t)=f(t), \\
	\prescript{}{0}{\mathcal{I}}^{(1-\nu_j)(1-\mu_{j})}_{\Psi(t)}y(0)&=b_{j}, \qquad j=1,2,3,
	\end{align*}
	has the solution
	\begin{align*}
	y(t) &= \sum_{k=0}^{\infty} \frac{(-1)^k}{a_3^{k+1}} \sum_{i=0}^{k} \binom{k}{i} {a_1}^i {a_2}^{k-i} \Bigg[ (\Psi(t))^{(\mu_{3}-\mu_{2})k+(\mu_{2}-\mu_{1})i+\mu_{3}-1} \\
	&\hspace{1.5cm}\times\Bigg(a_1 b_1 (\Psi(t))^{\nu_1 (1-\mu_{1})} E_{\mu_{3},(\mu_{3}-\mu_{2})k+(\mu_{2}-\mu_{1})i+\mu_{3}+\nu_1 (1-\mu_{1})} \left( -\frac{a_4}{a_3} (\Psi(t))^{\mu_{3}} \right) \\
	&\hspace{2cm}+ a_2 b_2 (\Psi(t))^{\nu_2 (1-\mu_{2})} E_{\mu_{3},(\mu_{3}-\mu_{2})k+(\mu_{2}-\mu_{1})i+\mu_{3}+\nu_2 (1-\mu_{2})} \left( - \frac{a_4}{a_3} (\Psi(t))^{\mu_{3}} \right) \\
	&\hspace{2cm}+  a_3 b_3 (\Psi(t))^{\nu_3 (1-\mu_{3})} E_{\mu_{3},(\mu_{3}-\mu_{2})k+(\mu_{2}-\mu_{1})i+\mu_{3}+\nu_3 (1-\mu_{3})} \left( - \frac{a_4}{a_3} (\Psi(t))^{\mu_{3}} \right) \Bigg) \\
	&\hspace{0.5cm}+\left((\Psi(t))^{(\mu_{3}-\mu_{2})k+(\mu_{2}-\mu_{1})i+\mu_{3}-1}E_{\mu_{3},(\mu_{3}-\mu_{2})k+(\mu_{2}-\mu_{1})i+\mu_{3}}^{k+1} \left( - \frac{a_4}{a_3} (\Psi(t))^{\mu_{3}} \right)\right) *_{\Psi} f(t) \Bigg]. \nonumber
	\end{align*}
\end{theorem}

\begin{proof}
	It is easy to derive the solution by combining the result of \cite[Theorem 5]{Z. Tomovski} with the operational calculus technique demonstrated in the previous proof. So we omit the straightforward details. 
\end{proof}

\section{Conclusions}

In this paper, we have studied generalisations of some important operations used in analysis and differential equations, namely the Laplace transform and also convolution. These generalisations are given by composition with an increasing function $\Psi$, in a way closely connected with the notion of differentiating and integrating with respect to $\Psi$.

We have proved several important properties of these operations by using an operational calculus approach. Armed with these tools and properties, we have given a sufficient condition to guarantee the effectiveness of the generalised Laplace transform in solving constant-coefficient fractional differential equations, and we have found analytic solutions of several generalised Cauchy problems in the settings of $ \Psi $-RL, $ \Psi $-Caputo and $ \Psi $-Hilfer fractional derivatives whose special cases have remarkable physical applications.

There are many potential ideas for future work in this direction. By extending the limits of integration to the entire real axis, the two-sided generalised Laplace transform can be defined. Moving forward in the same line of thought, several other generalised transforms can be developed, such as a generalised Fourier transform and generalised Mellin transform with respect to a function $\Psi(t)$. In future, these can also be applied to $\Psi$-fractional differential equations in the same way as the classical Fourier and Mellin transforms can be applied to classical FDEs. 

\end{document}